\documentclass[12pt, reqno]{amsart}
\usepackage{amsmath, amsthm, amscd, amsfonts, amssymb, graphicx, color, mathrsfs,enumitem}
\usepackage[all]{xy}
\usepackage{slashed}
\usepackage{tipa}
\usepackage{soul}
\usepackage{cancel}
\usepackage{ulem}
\usepackage{mathtools}

\usepackage{xcolor}
\usepackage[colorlinks,citecolor=blue,hypertexnames=false]{hyperref}

\newtheorem{theorem}{Theorem}[section]
\newtheorem{lemma}[theorem]{Lemma}

\theoremstyle{definition}
\newtheorem{definition}[theorem]{Definition}

\theoremstyle{remark}
\newtheorem{remark}[theorem]{Remark}  
\numberwithin{equation}{section}

\begin{document}
\setcounter{page}{1}

\title[Heat equations driven by mixed local-nonlocal operators]{Fujita-type results for the semilinear heat equations driven by mixed local-nonlocal operators}

\author[V. Kumar]{Vishvesh Kumar}
\address{
  Vishvesh Kumar:
  \endgraf
  Department of Mathematics: Analysis, Logic and Discrete Mathematics
  \endgraf Ghent University \endgraf Krijgslaan 281, Building S8,	B 9000 Ghent, Belgium
  \endgraf
  and 
  \endgraf
  Department of Mathematical Sciences \endgraf
		Indian Institute of Technology (BHU)  \endgraf
		Varanasi, Uttar Pradesh, 221005, India. \endgraf
  {\it E-mail address} {\rm  Vishvesh.Kumar@UGent.be/vishveshmishra@gmail.com}
  }

  \author[B. T. Torebek ]{Berikbol T. Torebek} \address{Berikbol T. Torebek  \endgraf  Institute of
Mathematics and Mathematical Modeling \endgraf 28 Shevchenko str.,
050010 Almaty, Kazakhstan.} \email{berikbol.torebek@ugent.be}

\dedicatory{Dedicated to Prof. Hiroshi Fujita on the occasion of 60 years of ``Fujita critical exponent"}

\subjclass[2020]{35K58, 35B33, 35A01, 35B44}

\keywords{mixed local-nonlocal operator; critical exponents; heat equation; global existence; blow-up}
\date{\today}
\thanks{The authors are supported by the FWO Odysseus 1 grant G.0H94.18N: Analysis and Partial Differential Equations, the Methusalem programme of the Ghent University Special Research Fund (BOF) (Grant number 01M01021). VK is also supported by  FWO Senior Research Grant G011522N. BT is also supported by the Science Committee of the Ministry of Education and Science of the Republic of Kazakhstan (Grant No. BR31714735).}

\begin{abstract} This paper explores the critical behavior of the semilinear heat equation $u_t+\mathcal{L}_{a, b}u=|u|^p+f(x)$, considering both the presence and absence of a forcing term  $f(x).$ The mixed local-nonlocal operator $\mathcal{L}_{a, b}=-a\Delta+b(-\Delta)^s,\,a,\,b \in \mathbb{R}_+,$ incorporates both local and nonlocal Laplacians. We determine the Fujita-type critical exponents by considering the existence or nonexistence of global solutions. Interestingly, the critical exponent is determined by the nonlocal component of the operator and, as a result, coincides with that of the fractional Laplacian. 

In the case without a forcing term, our results improve upon recent findings by Biagi et al. [Bull. London Math. Soc.  57 (2025), 265--284] and Del Pezzo et al. [Nonlinear Analysis 255 (2025), 113761]. When a forcing term is included, our results refine those of Wang et al. [J. Math. Anal. Appl., 488 (1) (2020), 124067] and complement the work of Majdoub [La Matematica, 2 (2023), 340–361].
 
\end{abstract} \maketitle

\tableofcontents
\allowdisplaybreaks
\section{Introduction and main results}
\subsection{Introduction} This paper is concerned with  the  semilinear heat equations driven by the {\it mixed  local-nonlocal} operator $\mathcal L_{a,b}$ of the following type:
\begin{equation} \label{main}
 \begin{cases}
 u_t + \mathcal L_{a,b} u = |u|^p + f(x) & \text{in $\mathbb{R}^d\times(0,+\infty)$} \\
 u(x,0) = u_0(x) & \text{in $\mathbb{R}^d$},
 \end{cases}
\end{equation}
where $p>1,$ $u_0(x)$ is a given initial data, and $f(x)$ is a forcing term. Our primary objective is to determine the critical exponent \eqref{main} by analyzing the blow-up and global existence of solutions, both with and without a forcing term. In addition to this, we will also investigate the existence and uniqueness  of the local solutions. 

The operator $$\mathcal L_{a,b} = -a\Delta+b(-\Delta)^s,$$  where $a,b\in\mathbb{R}_+,$ and $(-\Delta)^s$ stand for the (nonlocal) fractional Laplacian of order $s\in (0,1)$:
\begin{align*}
(-\Delta)^s v(x) & = C_{d,s}\cdot \mathrm{P.V.}\int_{\mathbb{R}^d}\frac{v(x)-v(y)}{|x-y|^{d+2s}}\,dy
\\
& = C_{d,s}\cdot\lim_{\varepsilon\to 0^+}\int_{\{|x-y|\geq\varepsilon\}}\frac{u(x)-u(y)}{|x-y|^{d+2s}}\,dy,\,C_{d,s}=\frac{2^{2s-1}{2s}\Gamma\left(\frac{d+2s}{2}\right)}{\pi^{d/2}\Gamma(1-s)},
\end{align*}
has been studied intensively in recent years in several different contexts by many prominent researchers. This is due to its numerous applications in various areas of mathematics, such as probability theory and mathematical biology. Specifically, in probability theory, the mixed operator arises from the superposition of a classical (that is, Brownian) stochastic process and a long-range (that is, L\'evy) stochastic process. In mathematical biology, mixed operators are used to model optimal animal foraging strategies. For further details, we refer to \cite{Chen1, Chen2, Dip1, Dip2, Dip3, Dip4, Filip} and the references therein.

In the classical paper  \cite{Fujita}, Fujita considered the problem \eqref{main} for $$f\equiv 0\,\,\, \text{and} \,\,\,\mathcal{L}_{1,0}=-\Delta,$$ and proved that: for any $u_0\geq 0$ the problem \eqref{main} possesses no global positive solutions if $$1<p<p_F=1+\frac{2}{d},$$ while there exists a positive global  solution of \eqref{main} if $p>p_F$ and $u_0\geq 0$ is smaller than a Gaussian. Later, Hayakawa \cite{Hayakawa} and Sugitani \cite{Sugitani} established that the $p=p_F$ also belongs to case blow-up case for nonnegative initial data. The number $p_F=1+\frac{2}{d}$ is called the {\it Fujita (critical) exponent}. In   \cite{Wes81}, Weissler proved that there exists a global solution of he problem \eqref{main} for $f\equiv 0$ and $\mathcal{L}_{1,0}=-\Delta,$ if $p>p_F$ and $\Vert u_0\Vert_{L^{q_c}(\mathbb{R}^d)}$ is sufficiently small with $$q_c=\frac{d(p-1)}{2}>1.$$ Zhang in \cite{Zhang} extended Fujita's results, by replacing $u_0\geq 0$ with $$\int_{\mathbb{R}^d}u_0(x)dx>0,$$ and proved that the Fujita's problem possesses no global sign-changing solutions if $p\leq p_F$. Later, Pinsky \cite{Pinsky} showed that Fujita's problem has no globally sign-changing solutions when $p\leq p_F$ and $$\int_{\mathbb{R}^d}u_0(x)dx=0,\,u_0\not\equiv 0.$$

On the other hand, many authors studied the problem \eqref{main} for $f\equiv 0$ and for the nonlocal fractional Laplacian $\mathcal{L}_{0,1}=(-\Delta)^s,\,s\in (0,1);$ we refer to  \cite{Fino, Kirane1, Naga, QS07, Sugitani} and reference therein for more details. It was observed in these studies that the critical exponent for the problem \eqref{main} for $f\equiv 0$ and  $\mathcal{L}_{0,1}$ in the sense of Fujita is $p_F:=1+\frac{2s}{d}.$

 Bandle et al. \cite{BLZ00} considered the problem \eqref{main} for $\mathcal{L}_{1,0}=-\Delta,$ and found a new critical exponent in the form $$p_c:=\left\{\begin{array}{cc}
   \infty,  & d=1,2, \\
   \frac{d}{d-2},  & d\geq 3.
\end{array}\right.$$ It was shown that if $1<p\leq p_c$ and $\int_{\mathbb{R}^d}f(x)dx>0,$ then there exists no global solutions, while if $p> p_c,$ then the problem has global solutions. Note that similar results for $f\geq 0$ were obtained a little earlier in \cite{Pinsky1, Zhang1, Zhang3, Zhang4}. Some further extensions and improvements of these results can be found in \cite{Torebek1, Falc, Jleli, Torebek2}.\\

In \cite{Wang}, the authors
studied the critical behavior of problem \eqref{main} with $\mathcal{L}_{0,1}=(-\Delta)^s,\,s\in (0,1),$ and
proved that $p_c=\frac{d}{d-2s}$ is the critical exponent  for $u_0\geq 0$ and $f\geq 0.$
Recently, in \cite{Majdoub} Majdoub studied the critical behavior of the problem \eqref{main} with $\mathcal{L}_{0,1}=(-\Delta)^s,\,s\in (0,1),$ and proved the existence of global (not necessarily positive) solutions for $p>\frac{d}{d-2s}.$ However, the behavior of sign-changing solutions for $p\leq \frac{d}{d-2s}$ has not been studied and was left open due to the technical shortcomings of the method used.

Recently, Biagi, Punzo, and Vecchi  \cite{Biagi} studied the problem \eqref{main} with $\mathcal{L}_{1,1}$ and $f\equiv 0$ and found the Fujita-type critical exponent $1+\frac{2s}{d}$ of the existence/non-existence of a positive global solution for $u_0\geq 0$. The same results, but for $\mathcal{L}_{a,b}$ were obtained by Del Pezzo and Ferreira in \cite{Pezzo}.
In \cite{Fino}, the authors obtained the asymptotic behavior of the solutions of the problem \eqref{main}, where instead of $|u|^p$ were considered $-h(t)u^p.$

Motivated by the above remarkable results, this paper investigates problem \eqref{main}, focusing on the critical behavior of sign-changing solutions. Specifically, when $f\equiv 0,$ we aim to extend the results of Biagi et al. \cite{Biagi} and Del Pezzo and Ferreira \cite{Pezzo} by removing the assumption $u_0\geq 0$. When $a>0, b\geq 0$ and $f\not\equiv 0$, our goal is to determine the critical exponent without assuming $u_0, f\geq 0$. This, in a particular case, generalizes the results of Wang and Zhang \cite{Wang} and complements the findings of Majdoub \cite{Majdoub}.

\subsection{Main results}\label{mr}\setcounter{equation}{0} We now present our main findings in detail. To this end, we first introduce the notion of a solution to the Cauchy problem \eqref{main}. Specifically, we consider two standard types of solutions, namely, weak and mild solutions, which are commonly used in the study of parabolic equations. This definition is motivated by the notions of weak and mild solutions in \cite{Biagi} for the case $f \equiv 0$, and we recommend it for further discussion. Here we mention that the authors of \cite{Biagi} refer to weak solutions as very weak solutions.

\begin{definition} \label{def:solWeakMild}
 Let $u_0\in L^1_{loc}(\mathbb R^d)\cap C_0(\mathbb R^d),$ $f\in L^1_{loc}(\mathbb R^d)\cap C_0(\mathbb R^d),$ and let $1\leq  p<+\infty$.
 \begin{itemize}
  \item[1)] (\textbf{Weak solution}) We say that a function $u$
 is a \textbf{weak solution} to problem \eqref{main} if $u\in L^p_{\mathrm{loc}}([0,+\infty)\times\mathbb{R}^d)$ and satisfies the following integral identity:
\begin{equation} \label{eq:weakSoldef}\begin{split}
\int_{0}^\infty\int_{\mathbb{R}^d}[|u|^p+f(x)]\varphi\,dx\,dt&+\int_{\mathbb R^d}u_0(x)\varphi(x,0)\,dx\\&= \int_{0}^\infty\int_{\mathbb{R}^d}u(-\varphi_t+\mathcal L_{a,b}\varphi)\,dx\,dt,\end{split}
  \end{equation} for any $\varphi\in C^\infty_0((0,\infty)\times\mathbb{R}^d).$
 \item[2)] (\textbf{Mild solution})
 We say that a function $u$
 is a
 \textbf{mild solution} to problem \eqref{main} if $u\in C([0,+\infty)\times\mathbb{R}^d)\cap L^\infty((0,+\infty)\times\mathbb{R}^d)$ and satisfies the following integral equation
  \begin{equation} \label{eq:mildSoldef}
   u(x,t) = \int_{\mathbb R^d}\mathcal{P}_t(x-y)u_0(y)\,dy
  + \int_0^t\int_{\mathbb R^d}\mathcal{P}_{t-\tau}(x-y)(|u(y,\tau)|^p+f(y))\,dy\,d\tau.
  \end{equation}
 \end{itemize}
\end{definition}

With the notion of solutions established, we are now ready to state our first main result on the existence and nonexistence of sign-changing solutions to \eqref{main} with  $f(x)\equiv 0.$ 

\begin{theorem} \label{thm:Main}
 Let $a,b>0,$ $s \in (0, 1),$ $f(x)\equiv 0,$ $u_0\in L^1_{loc}(\mathbb R^d),$ and let $1< p<\infty$.\\
 Then, the following facts hold.
 \begin{itemize}
  \item[(i)] If $$1< p\leq {p}_F= 1+\frac{2s}{d}\,\,\, \text{and}\,\,\,\int_{\mathbb{R}^d}u_0(x)dx>0,$$ 
then the Cauchy problem \eqref{main} admits no global weak solution.
  \vspace*{0.1cm}
  \item[(ii)] If $$1< p\leq {p}_F\,\,\, \text{and}\,\,\,\int_{\mathbb{R}^d}u_0(x)dx=0,\,u_0\not\equiv 0,$$ 
then the Cauchy problem \eqref{main} admits no global weak solution.
  \vspace*{0.1cm}
  \item[(iii)]Let $a=b=1$. If \,$p > {p}_F$, and $u_0 \in C_0(\mathbb{R}^{d}) \cap L^{p_c^s}(\mathbb{R}^{d}),$  where
  \begin{equation}
      p_c^s:=\frac{d(p-1)}{2s},
  \end{equation}
  with $\Vert u_0 \Vert_{L^{p_c^s}(\mathbb{R}^d)}$ sufficiently small, then the solution $u$ of \eqref{main}   exists globally. 
 \end{itemize}
\end{theorem}
\begin{remark} \begin{itemize}
    \item[(i)] As we have mentioned earlier, Theorem \ref{thm:Main} improves the results presented in   \cite[Theorem 3.3]{Biagi} and results in \cite[Theorem 1.3]{Pezzo}, since we do not assume the positivity of $u_0$. Furthermore, if $s=1$, then part (i) of Theorem \ref{thm:Main} coincide with the results of Zhang in \cite{Zhang}, and part (ii) of Theorem \ref{thm:Main} coincides with the results by Pinsky obtained in \cite{Pinsky}, for the classical heat equation. Also, when $s=1$, then part (iii) of Theorem \ref{thm:Main} coincides with the results of Weissler in \cite{Wes81}
    \item[(ii)]  We would like to point out that in part (iii), we have chosen $a=b=1$ solely for simplicity in the proof. However, this result holds for any $a>0, b>0$ in view of the heat kernel estimate of $\mathcal{L}_{a, b}$ given in \cite{Pezzo}.
    \item[(iii)] Unfortunately, Theorem \ref{thm:Main} does not cover case $$\int_{\mathbb{R}^d}u_0(x)dx<0.$$ Following Pinsky \cite{Pinsky}, who studied case $s=1$, we anticipate that when $\int_{\mathbb{R}^d}u_0(x)dx<0,\,p>1$, the Cauchy problem \eqref{main} admits a global solution. However, due to technical challenges, we leave this question open for future investigation.
\end{itemize}
\end{remark}
Our next results is regarding existence and nonexistence of global solution to the problem \eqref{main} with $f\not\equiv 0$.
\begin{theorem} \label{thm:Main1}
 Let $d> 2s,$ $s \in (0, 1),$  $u_0\in L^1(\mathbb R^d)\cap C_0(\mathbb{R}^d),$ $f\in L^1_{loc}(\mathbb R^d)\cap C_0(\mathbb R^d),$ and let $1< p<\infty$. Then the following facts hold.
 \begin{itemize}
  \item[(i)] If $a\geq 0, b>0,$ $$1< p< {p}_{Crit}= \frac{d}{d-2s}\,\,\, \text{and} \,\,\,\int_{\mathbb{R}^d}f(x)dx>0,$$ then the Cauchy problem \eqref{main} admits no weak global solution.
  \vspace*{0.1cm}
\item[(ii)] If $a\geq 0, b>0,$ $$1< p={p}_{Crit}= \frac{d}{d-2s}\,\,\, \text{and} \,\,\,\int_{\mathbb{R}^d}f(x)dx>0,$$ and $$\limsup_{R\rightarrow\infty}R^{-\sigma}\int_{|x|<R}f(x)\,dx>0,\,\,0<\sigma\leq d,$$  then the Cauchy problem \eqref{main} admits no global weak solution.
  \vspace*{0.1cm}
  \item[(iii)]  Let $a=b=1.$ If  
   \begin{equation} \label{eq2.15}
       p> {p}_{Crit}:=\frac{d}{d-2s},
   \end{equation}
   and 
    \begin{equation} 
         p_{c}^{s} =\frac{d(p-1)}{2s}, \quad k= \frac{p_{c}^{s}}{p}.
    \end{equation}
    Then for any $u_0 \in C_0(\mathbb{R}^{d}) \cap L^{p_{c}^{s}}(\mathbb{R}^d)$ and $0 \not\equiv  f$ with $\Vert u_0 \Vert_{L^{p_{c}^{s}}(\mathbb{R}^d)}+ \Vert f \Vert_{L^{k}(\mathbb{R}^d)}$ sufficiently small, the solution $u$ of \eqref{main}  exists globally. 
 \end{itemize}
\end{theorem}

\begin{remark} \begin{itemize}
    \item[(i)] The part (i) of Theorem \ref{thm:Main1} (when $a=0$) generalizes the results of Wang and Zhang \cite{Wang} and complements the findings of Majdoub \cite[Remark 1.13]{Majdoub}. If $s=1$,
then Theorem \ref{thm:Main1} coincides with the results of Bandle, Levine, and Zhang in \cite{BLZ00}, for the classical heat equation.
\item[(ii)] Part (ii) of the Theorem \ref{thm:Main1} contains an additional constraint on $f$ in the form of $$\limsup_{R\rightarrow\infty}R^{-\sigma}\int_{|x|<R}f(x)\,dx>0,\,\,0<\sigma\leq d.$$ Unfortunately, at the moment we are unable to remove this restriction due to the method used. However, if one assumes $f\geq 0$ and studies the properties of positive solutions $u$, then using the approaches applied in \cite{BLZ00, Wang}, one can obtain the results of part (ii) without additional restrictions on $f$.
\item[(iii)] 
Of particular interest is the study of Fujita-type critical exponent in case $$\int_{\mathbb{R}^d}f(x)dx=0,\,f\not\equiv 0,$$ which is not included in Theorem \ref{thm:Main1}. Our methods do not allow us to examine this case, leaving the question open for future research. To our knowledge, Bandle et al. \cite{BLZ00} also raised a similar problem for the classical heat equation $u_t-\Delta u=|u|^p+f(x),\,x\in \mathbb{R}^d$, which also remains unsolved.\end{itemize}
\end{remark}
Below we present the last main result regarding the nonexistence of a local-in-time (instantaneous blow-up) solution. This complements the existence and uniqueness result (Theorem \ref{local}) about a local-in-time weak solution discussed in Section \ref{sec2}.
\begin{theorem}\label{thm:Main3}
Let $s \in (0, 1),$ $u_0\in L^1(\mathbb R^d)\cap C_0(\mathbb{R}^d),$ $f\in L^1_{loc}(\mathbb R^d)\cap C(\mathbb R^d),$ and let $1< p<\infty$. If $a\geq 0, b>0,$ and $$\limsup_{R\rightarrow\infty}R^{-\sigma}\int_{|x|<R}f(x)\,dx>0,\,\sigma>d,$$  then the Cauchy problem \eqref{main} admits no local in-time weak solutions for any $T>0$, that is, there is instantaneous blow-up of the local weak solution of problem \eqref{main}.
\end{theorem}

Apart from the introduction, the organization of this paper is as follows: The next section is devoted to the theory of existence and uniqueness of local solutions to \eqref{main}. The proofs of Theorem \ref{thm:Main} and Theorem \ref{thm:Main1} will be presented in the final section.

\section{Existence, uniqueness, extension and continuity of local solutions} \label{sec2}
In this section, we examine the existence and uniqueness of local solutions to \eqref{main}.  For simplicity, the results of this section will be presented for $a=b=1;$ however, the conclusions remain valid for any $a>0, b>0.$  We also note that while some results in this section could be derived from more general semigroup theory, we provide complete proofs here for the sake of completeness.

Let us denote by $\mathcal{L},$ the operator $\mathcal{L}_{1, 1}.$  We recall the following properties of the heat semigroup $e^{-t \mathcal{L}}$ from \cite[Lemma 5]{KFA25} and \cite[Theorem 2.5]{Biagi}.  
 There exists a constant  $C>0$ such that for any $1 \leq q \leq r \leq \infty,$ we have 
\begin{equation} \label{pqbound}
    \Vert e^{-t \mathcal{L}} \varphi \Vert_{L^r( \mathbb{R}^d )} \leq C  t^{-\frac{d}{2s}\left( \frac{1}{q}-\frac{1}{r} \right)}\Vert \varphi \Vert_{L^q(\mathbb{R}^d)}, \quad t>0,
\end{equation}
for any $\varphi \in L^q(\mathbb{R}^d).$ Indeed, this estimate also hold for $\mathcal{L}_{a, b}$ in view of \cite[Equation $(8)$]{Pezzo}. In particular, we have, for $q \in [1, \infty],$ 
\begin{equation} \label{Linfbound}
    \Vert e^{-t \mathcal{L}} \varphi \Vert_{L^q( \mathbb{R}^d )} \leq C \Vert \varphi \Vert_{L^q(\mathbb{R}^d)}, \quad t>0.
\end{equation}

Let us define the local mild solution for the Cauchy problem \eqref{main}.
\begin{definition} [Local mild solution] We say that a function $u \in C([0, T], C_0(\mathbb{R}^d))$ is a
     {\it local mild solution} of the problem \eqref{main} if it satisfies
\begin{equation} \label{MS2}
    u(t)= e^{-t \mathcal{L}_{a, b}} u_0+\int_0^t e^{-(t-\tau)\mathcal{L}_{a, b}} (|u(\tau)|^p +f) d\tau,
\end{equation}
for any $t \in [0, T).$
\end{definition}
Now, we state the main result of this section. 
\begin{theorem} \label{local} Let $d > 2s, p>1, a=b=1.$ Assume that $u_0, f \in C_0(\mathbb{R}^d).$ Then the following assertions hold. 
\begin{itemize}
    \item[(i)] There exists  a unique mild solution $u \in C([0, T], C_0(\mathbb{R}^d))$ of the problem \eqref{main} for some $0<T<\infty.$
    \item[(ii)] The solution $u$ can be extended to a maximal interval $[0, T_{\max}),$ where $0<T_{\max} \leq \infty.$ Furthermore, if $T_{\max} <\infty,$ then $$ \lim_{t \rightarrow T_{\max}^{-}} \Vert u(t) \Vert_{L^{\infty}(\mathbb{R}^d)} = \infty.$$
    \item[(iii)]  Additionally, we assume that $u_0, f \in L^{r}(\mathbb{R}^d), \,\, r \in [1, \infty).$ Then, we have  
    $$ u \in C([0, T_{\max}); C_0(\mathbb{R}^d)) \cap C([0, T_{\max}); L^r(\mathbb{R}^d)).$$
    
\end{itemize}
    
\end{theorem}
\begin{proof} To prove the existence part of  assertion  $(i),$ we define the Banach space $M,$ for any arbitrary $T>0,$ by 
$$M:= \left\{ u \in C([0, T], C_0(\mathbb{R}^d)):\, \Vert u \Vert_{L^\infty ((0, T); L^\infty(\mathbb{R}^d))} \leq 2 \delta(u_0, f) \right\},$$
where $\delta(u_0, f)= \max \{ \Vert u \Vert_{L^\infty(\mathbb{R}^d)}, \Vert f \Vert_{L^\infty(\mathbb{R}^d)}  \}.$

We equipped the space $M$ with the distance given  by the norm of  $C([0, T];\,C_0({\mathbb{R}^d}))$, such that
\begin{equation}\label{DIS}
d(u,v):=\Vert u-v\Vert_{L^{\infty}({(0,T);\, L^{\infty}(\mathbb{R}^d}))},\,\,\,u,v\in M.     
\end{equation}
Next, for $u\in M$ we define the map
\begin{equation}
    (\Psi u)(t)= e^{-t\mathcal{L}}u_0 + \int_0^te^{-(t-\tau)\mathcal{L}}(|u(\tau)|^p+f)d\tau,\,\,\,0\leq t\leq T.
\end{equation}

We apply standard techniques to prove the existence of a unique local mild solution as a fixed point of $\Psi$ using the Banach fixed point theorem.\\

Let us first show that 
 $\Psi$ maps $M$ to $M.$
We take $u\in M$, then, in view of \eqref{Linfbound}, we obtain that  estimate
\begin{align*} \Vert (\Psi u)(t)\Vert_{L^\infty(\mathbb{R}^d)} & \leq \Vert e^{-t\mathcal{L}} u_0 \Vert_{L^{\infty}({\mathbb{R}^d})} +\biggl\Vert \int_0^te^{-(t-\tau)\mathcal{L}} (|u(\tau)|^p+f)d\tau \biggr\Vert_{L^{\infty}({\mathbb{R}^d})}\\&\leq \Vert u_0 \Vert_{L^{\infty}({\mathbb{R}^d})}+T\left\Vert u\right \Vert^p_{L^\infty((0,T);\,L^{\infty}({\mathbb{R}^d}))}+ T\Vert f \Vert_{L^{\infty}({\mathbb{R}^d})}\\&\leq (1+T)\delta(u_0,f)
+2^p T \delta^p (u_0,f),
\end{align*}
holds for all $0< t\leq T$. Thus, we deduce that that 
\begin{align}\label{LL1}\Vert \Psi u\Vert_{L^\infty((0, T); L^\infty(\mathbb{R}^d))}&\leq (1+T
+2^p  \delta^{p-1}(u_0,f) T)\delta(u_0,f). \end{align}
We choose $T>0$  sufficiently small so that
\begin{equation*}\label{TM}
T+2^pT\delta^{p-1}(u_0,f)\leq1,    
\end{equation*}
and, by \eqref{LL1}, we obtain 
$$\Vert \Psi u \Vert_{L^\infty((0, T); L^\infty(\mathbb{R}^d))}\leq2 \delta(u_0,f),$$ which gives $\Psi( M)\subset M$. This proves our claim that $\Psi: M \rightarrow M.$
\\

Next, we will show that $\Psi$ is a contraction mapping. For $u,v \in M,$  we have 
\begin{align} \label{laptop1}
\Vert (\Psi u)(t)-(\Psi v)(t)\Vert_{L^\infty(\mathbb{R}^d)}&=\int_0^t e^{-(t-s)\mathcal{L}} (|u(\tau)|^p-|v(\tau)|^p)ds \nonumber \\&\leq 2^{p-1}C(p)T\delta^{p-1}(u_0,f) \Vert u-v \Vert_{L^\infty((0,T);\,L^{\infty}({\mathbb{R}^d}))},
\end{align}
where we have used the inequality
\begin{equation}\label{QWE}
 ||u|^{p}-|v|^{p}|\leq C(p)|u-v|\left(|u|^{p-1}+|v|^{p-1}\right).
\end{equation}
Therefore, by choosing $T$ sufficiently small such that 
\begin{equation*}
  2^{p-1}C(p)T\delta^{p-1}(u_0,f)<1,
\end{equation*}
we get, from \eqref{laptop1}, that
\begin{align}
\Vert (\Psi u)(t)-(\Psi v)(t)\Vert_{L^\infty(\mathbb{R}^d)}\leq C \Vert u - v \Vert_{L^\infty((0,T);\,L^{\infty}({\mathbb{R}^d}))},
\end{align}
for some positive constant $C<1.$
This shows that $\Psi$ is a contraction mapping on $M$. The Banach fixed point theorem now ensures the existence of a mild solution to \eqref{main}.

Finally, we prove the uniqueness of the solution. If $u, v$ are two mild solutions in $M$ for some $T > 0$, then using \eqref{Linfbound} and \eqref{QWE}, we obtain
\begin{align*}\Vert u(\tau)-v(\tau) \Vert_{L^{\infty}({\mathbb{R}^d})} & =\int_0^t e^{-(t-\tau)\mathcal{L}} \Vert|u(\tau)|^p-|v(\tau)|^p\Vert_{L^{\infty}({\mathbb{R}^d})}d\tau \\& \leq C(p)\int_0^t \Vert u(\tau)-v(\tau) \Vert_{L^{\infty}({\mathbb{R}^d})}d\tau,\quad 0 \leq t \leq T.
\end{align*}
This estimate, together with Gr\"{o}nwall inequality, imply the uniqueness for solution.

Next, we proceed with the proof of assertion (ii). Given the uniqueness of the mild solution, we deduce that the solution to the problem \eqref{main} can be extended over a maximal interval $ [0, T_{\max})$, where $T_{\max}$ is given by
$$
T_{\max} = \sup\left\{ t>0 \, :\,\, \eqref{MS2}\,\, \text{admits a solution \,\,} u \in C([0,t];\,C_0(\mathbb{R}^d))\right\}. 
$$
Suppose that  $T_{\max}<\infty.$ If possible, let us assume that there exists $K>0$ such that
\begin{equation}\label{MT}
\Vert u(t)\Vert_{L^\infty(\mathbb{R}^d)}\leq K,\,\,\text{for}\,\,0\leq t<T_{\max}.     
\end{equation}
 We define, for $0< \tau'< T_{\max},$ the set
$$
\mathcal{N} := \left\{v \in C([0,\tau'];\, C_0(\mathbb{R}^d):\Vert v\Vert_{ \,L^{\infty}((0,\tau');\,L^\infty(\mathbb{R}^d))}\leq 2\delta(K,f)\right\}, 
$$where $\delta(K,f)=\max\{K,\Vert f \Vert_{L^\infty(\mathbb{R}^d)}\}$. \\
Let us choose $t_*$ such that ${T_{\max}}/{2} <t_*< T_{\max}.$
For a given $v\in\mathcal{N}$, we define

$$(\Lambda v)(t):= e^{-t\mathcal{L}} u(t_*)+\int_0^t e^{-(t-\tau)\mathcal{L}}(|v(\tau)|^p+f)d\tau,\,\,0\leq t\leq\tau'.$$

Again, we equipped $\mathcal{N}$ with the distance \eqref{DIS}, and we have $\Lambda v\in C([0,\tau'];\,C_0(\mathbb{R}^d))$ by similar argument applied for $M$ and $\Psi$ in the proof of assertion $(i).$

Moreover, from \eqref{Linfbound} and \eqref{MT}, it follows that
\begin{equation}\begin{split}\label{TTT} \Vert (\Lambda v)(t)\Vert_{L^\infty(\mathbb{R}^d)} &\leq \Vert e^{-t\mathcal{L}}u(t_*)\Vert_{L^{\infty}({\mathbb{R}^d})}+\Vert \int_0^t e^{-(t-\tau)\mathcal{L}}(|v(\tau)|^p+f)d\tau \Vert_{L^{\infty}({\mathbb{R}^d})}\\& \leq \Vert u(t_*) \Vert_{L^{\infty}({\mathbb{R}^d})}+\tau' \left\Vert v \right\Vert^p_{L^\infty((0,T);\,L^{\infty}({\mathbb{R}^d}))}+ \tau' \Vert f \Vert_{L^{\infty}({\mathbb{R}^d})}\\&\leq (1+\tau')\delta(K,f)+2^p\tau'\delta^p(K,f),\end{split}\end{equation} for all $0\leq t\leq\tau'.$
Let $\tau'>0$ be a sufficiently small constant which satisfies
\begin{equation}\label{UND}
\tau' +2^p\tau' \delta^{p-1}(K,f)\leq 1.    
\end{equation}
Thus, we deduce from \eqref{TTT} that
$$\Vert \Lambda v\Vert_{L^\infty((0, \tau'), L^\infty(\mathbb{R}^d))}  \leq 2 \delta(K,f),$$
which gives us $\Lambda(\mathcal{N})\subset\mathcal{N}$. In addition, following a similar argument as before, under condition \eqref{UND}, the mapping  $\Lambda:\mathcal{N}\to\mathcal{N}$ is a contraction. By the Banach fixed point theorem, it follows that there exists a unique function  $v\in\mathcal{N}$ which satisfies
$$v(t)=e^{-t\mathcal{L}} u(t_*)+\int_0^t e^{-(t-s)\mathcal{L}}(|v(\tau)|^p+f)d\tau,\,\,0\leq t \leq \tau'.$$
At this point, for $\max\{T_{\max}/2;T_{\max}-\tau'\}<\tilde{t}<T_{\max},$ we set
\begin{equation*} \tilde{u}(t)=\left\{\begin{array}{l}
u(t)\,\,\,\text{if}\,\,\,0\leq t\leq\tilde{t}, \\{}\\
v(t-\tilde{t})\,\,\,\text{if}\,\,\,\tilde{t}\leq t\leq\tilde{t}+\tau'. \end{array}\right.\end{equation*}
Then we note that $\tilde{u}\in C([0,\overline{t}+\tau'];\,C_0(\mathbb{R}^d))$ is a solution to \eqref{MS2} and $T_{\max}<\tilde{t}+\tau'$, which contradicts the definition of $T_{\max}$. So the hypothesis \eqref{MT} does not hold. Hence, we deduce that if $T_{\max}<\infty$ then $\lim_{t\to T_{\max}^-}\Vert u(t)\Vert_{L^\infty{(\mathbb{R}^d})}=\infty$. It completes the proof.

Now, we present the proof of part (iii). By hypothesis, we have  $u_0, f\in L^r(\mathbb{R}^d)$ with $r \in [1, \infty).$ By repeating the
fixed point argument in the Banach space
\begin{equation*}\begin{split}
M_r=&\{u\in C([0, T];\,C_0({\mathbb{R}^d}))\cap C([0, T_{\max});\,L^r({\mathbb{R}^d})):\\& \quad \Vert u \Vert_{L^{\infty}({(0,T);\, L^{\infty}(\mathbb{R}^d}))}\leq 2 \delta(u_0,f),\,\Vert u \Vert_{L^{\infty}((0,T);\,L^{r}({\mathbb{R}^d}))}\leq2 \delta_r(u_0,f)\},    
\end{split}\end{equation*}
instead of $M$, where $\delta_r(u_0,f)=\max\{\Vert u_0\Vert_{L^{r}({\mathbb{R}^d})}; \Vert f \Vert_{L^{r}({\mathbb{R}^d})}\}.$
\\In addition, we endow $M_r$ with norm given by  the distance  
\begin{equation*}
d(u,v)=\Vert u-v\Vert_{L^\infty((0,T);\,L^\infty(\mathbb{R}^d))}+\Vert u-v\Vert_{L^\infty((0,T);\,L^r(\mathbb{R}^d))},\,\,\,u,v\in M_r.
\end{equation*}
Finally, applying the same argument as in the proof of the assertion $(i)$ using the inequality
$$\Vert|u(t)|^p\Vert_{L^r(\mathbb{R}^d)}\leq \Vert u(t)\Vert^{p-1}_{L^\infty(\mathbb{R}^d)}\Vert u(t)\Vert_{L^r(\mathbb{R}^d)},$$
 we obtain a unique solution $u$ in $M_r$. Therefore, we conclude that 
$$u\in C([0, T_{\max});\,C_0({\mathbb{R}^d}))\cap C([0, T_{\max});\,L^r({\mathbb{R}^d})),$$which completes the proof of the theorem.\end{proof}

Next, we will show that a mild solution is also a weak solution within our framework.
Let us first introduce the notion of a local weak solution to the Cauchy problem \eqref{main}.  
\begin{definition}[Local weak solution]\label{DWS1} Let $u_0, f\in L^1_\text{loc}(\mathbb{R}^d)$. A locally integrable function $u\in L^p_\text{loc}((0,T);\,L^p_\text{loc}(\mathbb{R}^d))$ is called a {\it local weak solution} of \eqref{main}, if the equality
\begin{equation} \label{eq:loweakSoldef}\begin{split}
\int_{0}^T\int_{\mathbb{R}^d}[|u|^p+f(x)]\varphi\,dx\,dt&+\int_{\mathbb R^d}u_0(x)\varphi(x,0)\,dx \\&= \int_{0}^T \int_{\mathbb{R}^d}u(-\varphi_t+\mathcal L_{a,b}\varphi)\,dx\,dt,\end{split}
  \end{equation} 
holds true for any function $\varphi \in C^{1}_{0}((0,T);\,C^{\infty}_{0}(\mathbb{R}^d)), \varphi\geq0$ and $\varphi (T,\cdot)=0.$

\end{definition}

\begin{lemma}\label{MWSL}
Suppose that $u_0\in C_0(\mathbb{R}^d)$ and $f\in C_0(\mathbb{R}^d).$ Consider that $u\in C([0,T];\,C_0(\mathbb{R}^d))$ is a mild solution of \eqref{main}. Then $u$ is also a weak solution to \eqref{main}.\end{lemma}

\begin{proof}
Take $T>0.$ Then, we multiply the identity \eqref{MS2} by $\varphi\in C^{1}_{0}((0,T);\,C^{\infty}_{0}(\mathbb{R}^d)),$ $\varphi(T,\cdot)=0$, and integrate over $\mathbb{R}^d$, to obtain
\begin{equation*}\label{MS3}\begin{split}
\int_{\mathbb{R}^d}u\varphi dx =\int_{\mathbb{R}^d}e^{-t\mathcal{L}_{a, b}}u_0\varphi dx +\int_{\mathbb{R}^d}\biggl(\int_0^te^{-(t-\tau)\mathcal{L}_{a, b}}(|u(\tau)|^p+f)d\tau\biggr)\varphi dx. 
\end{split}\end{equation*}
By differentiating the last equality with respect to the variable $t$, it follows that
\begin{equation}\label{MS4}\begin{split}
\frac{d}{dt}\int_{\mathbb{R}^d}u\varphi dx &=\int_{\mathbb{R}^d}\frac{d}{dt}\biggl(e^{-t\mathcal{L}_{a, b}}u_0\varphi \biggr)dx\\&+\int_{\mathbb{R}^d}\frac{d}{dt}\biggl[\biggl(\int_0^te^{-(t-\tau)\mathcal{L}_{a, b}}(|u(\tau)|^p+f)d\tau\biggr)\varphi \biggr]dx. 
\end{split}\end{equation}
Next, using the self-adjointness of $\mathcal{L}_{a, b}$, we deduce that
\begin{equation*}\label{MS5}\begin{split}
\int_{\mathbb{R}^d}\frac{d}{dt}\biggl(e^{-t\mathcal{L}_{a, b}}u_0\varphi \biggr)dx &=\int_{\mathbb{R}^d}(-\mathcal{L}_{a, b})e^{-t\mathcal{L}_{a, b}}u_0\varphi dx+\int_{\mathbb{R}^d}e^{-t\mathcal{L}_{a, b}}u_0\varphi_t dx
\\&=\int_{\mathbb{R}^d}e^{-t\mathcal{L}_{a, b}}u_0(-\mathcal{L}_{a, b}\varphi) dx+\int_{\mathbb{R}^d}e^{-t\mathcal{L}_{a, b}}u_0\varphi_t dx, 
\end{split}\end{equation*}
and 
\begin{equation*}\label{MS6}\begin{split}
\int_{\mathbb{R}^d}\frac{d}{dt}\biggl[\biggl(\int_0^t&e^{-(t-\tau)\mathcal{L}_{a, b}}(|u(\tau)|^p+f)d\tau\biggr)\varphi \biggr]dx\\&=\int_{\mathbb{R}^d}|u(t)|^p\varphi dx+\int_{\mathbb{R}^d}f\varphi dx\\&+\int_{\mathbb{R}^d}\biggl(\int_0^te^{-(t-\tau)\mathcal{L}_{a, b}}(|u(\tau)|^p+f)d\tau\biggr)(-\mathcal{L}_{a, b}\varphi) dx \\&+\int_{\mathbb{R}^d}\biggl(\int_0^te^{-(t-\tau)\mathcal{L}_{a, b}}(|u(\tau)|^p+f)d\tau\biggr)\varphi_t dx. 
\end{split}\end{equation*}
Therefore, taking into account \eqref{MS2} and the last identities, the equality \eqref{MS4} can be further rewritten as
\begin{equation*}\label{MS7}\begin{split}
\frac{d}{dt}\int_{\mathbb{R}^d}u\varphi dx &=\int_{\mathbb{R}^d}u(-\mathcal{L}_{a, b}\varphi) dx+\int_{\mathbb{R}^d}u\varphi_t dx+\int_{\mathbb{R}^d}|u(t)|^p\varphi dx +\int_{\mathbb{R}^d}f\varphi dx. \end{split}\end{equation*}

Finally, by integrating it in time variable over $[0,T]$ and using $\varphi(T,\cdot)=0,$ we obtain the desired result, completing the proof. \end{proof}

\section{Proof of Main results} 
\label{sec3}
This section is devoted to presenting the proof of the main results, namely, Theorem \ref{thm:Main} and Theorem \ref{thm:Main1}. We first prove Theorem  \ref{thm:Main1}. 
\subsection{Proof of Theorem \ref{thm:Main1}}\label{SubsectionP1}  The proof will be divided into three different cases.  \subsubsection{The case $p<p_{Crit}$}\label{SSS1} The proof will be given by contradiction. Suppose that there exists a non-trivial
global solution of Cauchy problem \eqref{main}. Then from the definition of weak solution \eqref{eq:weakSoldef} we have
\begin{equation}\label{proof11}\begin{split}
\int_{0}^\infty\int_{\mathbb{R}^d}&[|u|^p+f(x)]\varphi\,dx\,dt+\int_{\mathbb R^d}u_0(x)\varphi(x,0)\,dx\\&\leq {\int_{0}^\infty\int_{\mathbb{R}^d}|u||\varphi_t|\,dx\,dt}+{\int_{0}^\infty\int_{\mathbb{R}^d}|u||\mathcal L_{a,b}\varphi|\,dx\,dt}.\end{split}
  \end{equation}
The test function $\varphi$ is chosen so that
\begin{equation}\begin{split}
\int_{\mathbb{R}_+}\int_{\mathbb{R}^d}\varphi^{-\frac{1}{p-1}}\left|\varphi_t\right|^{\frac{p}{p-1}}dx dt<\infty, \\ \int_{\mathbb{R}_+}\int_{\mathbb{R}^d}\varphi^{-\frac{1}{p-1}}\left|\varphi\right|^{\frac{p}{p-1}}dx dt<\infty.
\end{split}\end{equation}
By using the $\varepsilon$-Young inequalities to the right hand side of \eqref{proof11} we have
\begin{align*}
\int_{\mathbb{R}_+}\int_{\mathbb{R}^d}|u||\varphi_t|\,dx\,dt& \leq \varepsilon\int_{\mathbb{R}_+}\int_{\mathbb{R}^d}|u|^p\varphi\,dx\,dt \\&+C(\varepsilon)\int_{\mathbb{R}_+}\int_{\mathbb{R}^d}|\varphi_t|^{\frac{p}{p-1}}\varphi^{-\frac{1}{p-1}}\,dx\,dt
\end{align*}
and
\begin{align*}
\int_{\mathbb{R}_+}\int_{\mathbb{R}^d}|u||\mathcal{L}_{a,b}\varphi|\,dx\,dt& \leq \varepsilon\int_{\mathbb{R}_+}\int_{\mathbb{R}^d}|u|^p\varphi\,dx\,dt\\& +C(\varepsilon)\int_{\mathbb{R}_+}\int_{\mathbb{R}^d}|\mathcal{L}_{a,b}\varphi|^{\frac{p}{p-1}}\varphi^{-\frac{1}{p-1}}\,dx\,dt,
\end{align*} where $C(\varepsilon)>0.$

Choosing $2\varepsilon<1$ and substituting the above calculations into \eqref{proof11} we have
\begin{equation}\label{proof1}\begin{split}
(1-2\varepsilon
)\int_{0}^\infty\int_{\mathbb{R}^d}|u|^p\varphi\,dx\,dt&+\int_{0}^\infty\int_{\mathbb{R}^d}f(x)\varphi\,dx\,dt+\int_{\mathbb R^d}u_0(x)\varphi(x,0)\,dx\\&\leq C\underbrace{\int_{\mathbb{R}_+}\int_{\mathbb{R}^d}|\varphi_t|^{\frac{p}{p-1}}\varphi^{-\frac{1}{p-1}}\,dx\,dt}_{\mathcal{I}_1}\\&+C\underbrace{\int_{\mathbb{R}_+}\int_{\mathbb{R}^d}|\mathcal{L}_{a,b}\varphi|^{\frac{p}{p-1}}\varphi^{-\frac{1}{p-1}}\,dx\,dt}_{\mathcal{I}_2}.\end{split}
\end{equation}
We set the test function as $$\varphi(t,x)=\eta(t)\phi(x),$$
for large enough $R, T,$ where
\begin{equation}\label{T2}
\eta(t)=\nu^l\left(\frac{t}{T}\right),\,l=\frac{2p}{p-1},\,\,\, t>0, \end{equation}
and
\begin{equation}\label{T1}
\phi(x)=\Psi^l\left({\frac{|x|}{{R}}}\right),\,l=\frac{2p}{p-1},\,\, x\in\mathbb{R}^{d}.
\end{equation}
Let $\nu\in C^\infty(\mathbb{R}_+)$ be such that
$\nu\equiv 1$ on $[0,1/2]$; $\nu\equiv 0$ on $[1,\infty),$ and $\Psi\in C^\infty(\mathbb{R}^d)$ be a smooth function  satisfying
\begin{equation}\label{TT}
\Psi(s)=\left\{\begin{array}{l}
1,\,\,\,\,\,\text{if}\,\,0\leq s\leq 1/2,\\
\searrow,\,\,\text{if}\,\,1/2< s<1,\\
0,\,\,\,\,\,\text{if}\,\,s\geq1.\end{array}\right.
\end{equation}
Then using the above properties of function $\varphi$ we have
\begin{equation}\label{00}\begin{split}
\mathcal{I}_1&=\int_{T/2}^T\int_{|x|<R}|\varphi_t|^{\frac{p}{p-1}}\varphi^{-\frac{1}{p-1}}\,dx\,dt\\& =\int_{T/2}^T|\eta'(t)|^{\frac{p}{p-1}}(\eta(t))^{-\frac{1}{p-1}}dt\int_{|x|<R}\phi(x)\,dx\\&\leq R^d \int_{T/2}^T\left|\nu'(t/T)\right|^{\frac{p}{p-1}}(\nu^{l-1}(t/T))^{-\frac{p}{p-1}}(\nu^l(t/T))^{-\frac{1}{p-1}}dt\\&\leq CR^dT^{1-\frac{p}{p-1}}.\end{split}
\end{equation}
Elementary calculations gives us
$$\Delta\phi = lR^{-2}\left[\Psi^{l-1}\left(\frac{|x|}{R}\right)\Delta\Psi\left(\frac{|x|}{R}\right)+(l-1)\Psi^{l-2}\left(\frac{|x|}{R}\right)|\nabla\Psi|^2\left(\frac{|x|}{R}\right)\right].$$
Since the function $G(z) = z^l$ is convex, by Cordoba-Cordoba-type inequality \cite[Lemma\,3.2]{PV1} we have
\begin{equation*}
\begin{split} (-\Delta)^s\phi & = (-\Delta)^s\left(G\circ \left(x\mapsto\Psi\left(\frac{|x|}{R}\right)\right)\right)
  \\& \leq l\Psi^{l-1}\left(\frac{|x|}{R}\right)(-\Delta)^s\left(x\mapsto\Psi\left(\frac{|x|}{R}\right)\right) \\
   & = lR^{-2s}\Psi^{l-1}\left(\frac{|x|}{R}\right)(-\Delta)^s\Psi\left(\frac{|x|}{R}\right).
   \end{split}
  \end{equation*}
As $\phi\in C^\infty_0(\mathbb{R}^d),$ then (see \cite[Formula (3.4)]{PV1}) $$|(-\Delta)^s\phi|\leq CR^{-2s}.$$ Hence,
\begin{align*}
\left|\mathcal{L}_{a,b}\phi\right|^{\frac{p}{p-1}}&=\left|-a\Delta\phi+b(-\Delta)^s\phi\right|^{\frac{p}{p-1}}\\&\leq C\left(\left|\Delta\phi\right|^{\frac{p}{p-1}}+\left|(-\Delta)^s\phi\right|^{\frac{p}{p-1}}\right) \\&\leq C\left(R^{-\frac{2p}{p-1}}+R^{-\frac{2sp}{p-1}}\right)\phi^{\frac{1}{p-1}}.
\end{align*}
Using the above calculations and properties of function $\varphi$ one can get
\begin{align*}
\mathcal{I}_2&=\int_0^T\int_{|x|<R}|\mathcal{L}_{a,b}\varphi|^{\frac{p}{p-1}}\varphi^{-\frac{1}{p-1}}\,dx\,dt\\& =\int_0^T\eta(t)dt\int_{|x|<R}|\mathcal{L}_{a,b}\phi(x)|^{\frac{p}{p-1}}(\phi(x))^{-\frac{1}{p-1}}\,dx\\&\leq CTR^{d-\frac{2p}{p-1}}+CTR^{d-\frac{2sp}{p-1}}.
\end{align*}
Then \eqref{proof1} can be rewritten as
\begin{equation}\label{proof2}\begin{split}
\int_{0}^T\eta(t)\,dt\int_{|x|<R}f(x)\phi(x)\,dx&+\int_{|x|<R}u_0(x)\phi(x)\,dx\\&\leq CR^dT^{-\frac{1}{p-1}}+CTR^{d-\frac{2sp}{p-1}}.\end{split}
\end{equation}
It is obvious that $\int_{0}^T\eta(t)\,dt=CT.$ Hence, choosing $T = R^{2s}$ in \eqref{proof2}, we can verify that
\begin{equation}\label{proof3}\begin{split}
\int_{|x|<R}f(x)\phi(x)\,dx\leq R^{-2s}\int_{|x|<R}|u_0(x)|\,dx+CR^{d-\frac{2sp}{p-1}}.\end{split}
\end{equation} If $p<p_{Crit},$ then $d-\frac{2sp}{p-1}<0.$ Therefore, passing to the limit as $R\to\infty$ in \eqref{proof3} we deduce that $$\int_{\mathbb{R}^d}f(x)\,dx\leq 0,$$ which is a contradiction with $\int_{\mathbb{R}^d}f(x)\,dx>0.$

\subsubsection{The case $p=p_{Crit}$} 
Let $p=\frac{d}{d-2s}$, then \eqref{proof3} can be rewritten in the form
\begin{equation*}\begin{split}
\int_{|x|<{R}}f(x)\phi(x)\,dx\leq R^{-2s}\int_{|x|<R}|u_0(x)|\,dx+C.\end{split}\end{equation*}
Multiplying both side of the last expression by $R^{-\sigma},\,0<\sigma\leq d$ we have
\begin{equation*}\begin{split}
R^{-\sigma}\int_{|x|<{R}}f(x)\phi(x)\,dx\leq R^{-2s-\sigma}\int_{|x|<R}|u_0(x)|\,dx+CR^{-\sigma}.\end{split}\end{equation*}
Passing to the limit as $R\to\infty$ in the last inequality we deduce that $$\limsup_{R\rightarrow\infty}R^{-\sigma}\int_{|x|<R}f(x)\,dx\leq 0,$$ which is a contradiction with $$\limsup_{R\rightarrow\infty}R^{-\sigma}\int_{|x|<R}f(x)\,dx>0,\,\,0<\sigma\leq d.$$

\subsubsection{The case $p>p_{Crit}$} 
For the fixed $p>\frac{d}{d-2s},$ we choose $q$ such that 
    \begin{equation} \label{eq2.16}
        \frac{2s}{dp(p-1)}<\frac{1}{q}< \frac{2s}{d(p-1)}.
    \end{equation}
    Thus it follows that 
    \begin{equation} \label{qnk}
        q>p_{c}^{s} >k \geq 1,
    \end{equation} and for this fixed $q,$ we set 
    $$\rho := \frac{1}{p-1}- \frac{d}{2qs}>0,$$
    which implies that, for $s \in (0, 1]$ we have  
    \begin{equation} \label{eq2.18}
        \rho-\frac{d}{2s} \left( \frac{1}{d}-\frac{1}{q} \right)= (1-s)\rho.
    \end{equation}
    Let $\delta$ be a sufficiently small positive constant. We define the set 
    $$ \Theta:= \left\{ u \in L^\infty((0, \infty), L^q(\mathbb{R}^d)): \sup_{t>0} t^{\rho} \Vert u(t)\Vert_{L^q(\mathbb{R}^d)} \leq \delta  \right\}.$$
    Then space $\Theta$ equipped with the distance 
    $$d(u, v)= \sup_{t>0} t^{\rho} \Vert u(t)-v(t)\Vert_{L^q(\mathbb{R}^d)},\quad u, v \in \Theta,$$
    is a complete metric space. Given $u \in \Theta,$ set 
    \begin{equation} \label{eq2.19}
         (\Lambda u)(t):= e^{-t \mathcal{L}} u_0+\int_0^t e^{-(t-\tau)\mathcal{L}} (|u(\tau)|^p +f) d\tau, \quad t \geq 0.
    \end{equation}
    Since $u_0 \in L^{p_{c}^{s}}(\mathbb{R}^d)$ by \eqref{pqbound} and \eqref{eq2.18}, we have 
   \begin{align} \label{eq2.20}
        \Vert e^{-t \mathcal{L}} u_0 \Vert_{L^q(\mathbb{R}^d)} &\leq C  t^{-\frac{d}{2s}\left( \frac{1}{p_{c}^{s}}-\frac{1}{q} \right)}  \Vert u_0 \Vert_{L^{p_{c}^{s}}(\mathbb{R}^d)} \nonumber \\& = C  t^{- \rho}  \Vert u_0 \Vert_{L^{p_{c}^{s}}(\mathbb{R}^d)}  \, t>0,
   \end{align}
   for some positive constant $C$. Furthermore, applying the estimate \eqref{pqbound} for $(q, \frac{q}{p})$ to get 

   \begin{align} \label{eq2.21}
       \int_0^t \Vert e^{-(t-\tau) \mathcal{L}} |u(\tau)|^p \Vert_{L^q(\mathbb{R}^d)} &\leq C \int_0^t  (t-\tau)^{-\frac{d}{2qs}(p-1)}  \Vert |u(s)|^p \Vert_{L^{\frac{q}{p}}(\mathbb{R}^d)} d\tau
    \nonumber   \\& \leq C \delta^p \int_0^t \tau^{-\rho p}  (t-\tau)^{-\frac{d}{2qs}(p-1)}  d\tau \nonumber \\ & \leq 
    C \delta^p t^{-\frac{d}{2qs}(p-1)+1-\rho p} \nonumber \\& \quad \times   B\left( 1-\rho p, 1-\frac{d}{2qs}(p-1) \right)   \nonumber 
   \\& = C' \delta^p  t^{- \rho}, \quad t>0,
   \end{align}
   where $B$ denotes the beta function. We note $1-\frac{d}{2qs}(p-1)>0$ as $$1-\frac{d}{2qs}(p-1)=\rho(p-1)>0.$$ Therefore, the beta function is well defined.
   \\
   In a similar manner, we obtain that 
   \begin{align} \label{eq2.22}
       \int_0^t \Vert  e^{-(t-\tau)\mathcal{L}} f  \Vert_{L^q(\mathbb{R}^d)} d\tau & \leq C \Vert f \Vert_{L^k(\mathbb{R}^d)} \int_0^t  (t-\tau)^{-\frac{d}{2s}\left( \frac{1}{k}-\frac{1}{q} \right)}  d\tau \nonumber \\ & \leq  t^{-\frac{d}{2s}\left( \frac{1}{k}-\frac{1}{q} \right)+1}  B\left(1, 1-\frac{d}{2s} \left(\frac{1}{k}-\frac{1}{q} \right) \right) \nonumber \\&
      \leq  C'  t^{- \rho} \Vert f\Vert_{L^k(\mathbb{R}^d)}, \quad t>0,
   \end{align}
   where we have used the fact that $\rho-\frac{d}{2s} \left(\frac{1}{k}-\frac{1}{q} \right)+1 =0$ and the beta function is well-defined. Then, by \eqref{eq2.19}, \eqref{eq2.20}, \eqref{eq2.21} and \eqref{eq2.22}, we get
   $$ t^\rho \Vert (\Lambda u)(t) \Vert_{L^q(\mathbb{R}^d)} \leq C^* (\Vert u_0 \Vert_{L^{p_{c}^{s}}(\mathbb{R}^d)}+\delta^p+ \Vert f \Vert_{L^k(\mathbb{R}^d)}),\quad t>0,$$
   where $C^*>0$ is a positive constant, independent of $\delta.$ Thus, by choosing a sufficiently small positive constant $\delta$ satisfying 
   $$0< \delta \leq \left( \frac{1}{2 C^*}\right)^{\frac{1}{p-1}},$$
   and if the initial data $u_0$ and forcing term $f$ satisfy
   $$\Vert u_0 \Vert_{L^{p_{c}^{s}}}(\mathbb{R}^d)+ \Vert f \Vert_{L^{k}}(\mathbb{R}^d) \leq \frac{\delta}{2C^*},$$
   we get
   $$C^* (\Vert u_0 \Vert_{L^{p_{c}^{s}}(\mathbb{R}^d)}+\delta^p+ \Vert f \Vert_{L^k(\mathbb{R}^d)}) \leq \delta.$$
   This shows $\Lambda (\Theta) \subset \Theta.$ By applying a similar argument and choosing $\Vert u_0 \Vert_{L^{p_{c}^{s}}}(\mathbb{R}^d)+ \Vert f \Vert_{L^{k}}(\mathbb{R}^d)$ and $\delta$ small enough, if necessary, we can easily show that $\Lambda: \Theta \rightarrow \Theta$ is a contraction. Therefore, by Banach fixed point theorem, $\Lambda$ admits a fixed point $u \in L^\infty((0, \infty), L^q(\mathbb{R}^{d})),$ which is a mild solution of \eqref{main}.
   
Finally, we will prove that 
\begin{equation} \label{eq2.24}
    u \in C([0, \infty), C_0(\mathbb{R}^d)).
\end{equation}
The first step to prove this is to show that $u \in C([0, T), C_0(\mathbb{R}^d))$ for a sufficiently small time $T>0.$ For this small time $T$, we observe that the above argument gives uniqueness in the space
$$\Theta_T:= \left\{ u \in L^\infty((0, T), L^q(\mathbb{R}^d)):\,\,\, \sup_{0<t \leq T} t^{\rho} \Vert u\Vert_{L^q(\mathbb{R}^d)} \leq \delta \right\}.$$
Denote by $w$ the local solution of \eqref{main} obtained by the Theorem 
   \ref{local}. Using \eqref{qnk}, it follows that $u_0, f \in L^q(\mathbb{R}^d),$ we have $w \in u \in C([0, T_{\max}); L^\infty(\mathbb{R}^d)) \cap C([0, T_{\max}); L^r(\mathbb{R}^d)),$ thanks to Theorem \ref{local} (iii). Then, utilizing the boundedness of $\Vert w \Vert_{L^q(\mathbb{R}^d)},$ for a small enough $T>0,$ we deduce that $\sup_{0<t<T} \Vert w \Vert_{L^q(\mathbb{R}^d)} \leq \delta.$ Therefore, by the uniqueness of solutions, it follows that $w=u$ in $[0, T],$ and so that 
   \begin{equation} \label{eq2.25}
       u \in C([0, T], C_0(\mathbb{R}^d)).
   \end{equation}
   Now, we apply a bootstrap argument to show that $u \in C([ T, \infty), C_0(\mathbb{R}^d)).$ Indeed, for $t>T,$ it holds that 
   \begin{align*}
       u(t)-e^{-t \mathcal{L}} u_0 -\int_0^{t} e^{-(t-\tau) \mathcal{L}} f d\tau &= \int_0^T e^{-(t-\tau) \mathcal{L}} |u(\tau)|^p d\tau + \int_T^t e^{-(t-\tau) \mathcal{L}} |u(\tau)|^p d\tau \\&:= I_1(t)+I_2(t).
   \end{align*}
   Since $u \in C([0, T], C_0(\mathbb{R}^d)),$ it follows that  $I_1(t) \in C([T, \infty), C_0(\mathbb{R}^d)).$ Also, by calculations use to construct the fixed point $I_1 \in C([T, \infty), L^q(\mathbb{R}^d)).$ Next, note that $q> \frac{d(p-1)}{2s}$ by \eqref{eq2.16}. Therefore, there exists $q <r \leq \infty$ such that 
   $$\frac{d}{2s}\left(\frac{p}{q}-\frac{1}{r} \right)<1.$$
   Since $u \in L^\infty((0, \infty), L^q(\mathbb{R}^d)),$ for $\Tilde{T}>T,$ we know that $|u|^p \in L^\infty((0, \infty), L^{\frac{q}{p}}(\mathbb{R}^d)),$ and it easily follows that $I_2 \in C([T, \Tilde{T}], L^r(\mathbb{R}^d)).$ Since $\Tilde{T}$ is arbitrary, we get that $I_2 \in C([T, \infty),L^r(\mathbb{R}^d) ).$ Since the terms $e^{-t \mathcal{L}}u_0,$ $\int_0^t e^{-t \mathcal{L}} f d\tau $ and $I_1$ belong to $C([T, \infty), C_0(\mathbb{R}^d)) \cap C([T, \infty), L^q(\mathbb{R}^d),$ we deduce that $u \in C([T, \infty), L^r(\mathbb{R}^d)).$ Irerating this procdess a finite number of times, we obtain
   \begin{equation} \label{eq2.26}
       u \in C([T, \infty), C_0(\mathbb{R}^d)).
   \end{equation}
   Therefore, by combining \eqref{eq2.25} and \eqref{eq2.26}, we conclude the proof of our claim \eqref{eq2.24} and hence, completing the proof of the theorem. 

\subsection{Proof of Theorem \ref{thm:Main}}\subsubsection{The case $p<p_{F}$ and $\int_{\mathbb{R}^d}u_0(x)\,dx>0$} The proof is carried out in the same way as the case with $f(x)\not\equiv 0$ (see Subsection \ref{SubsectionP1}). Let us rewrite inequality \eqref{proof2} for $f(x)\equiv 0:$
\begin{equation*}\begin{split}
\int_{|x|<R}u_0(x)\phi(x)\,dx&\leq CR^dT^{-\frac{1}{p-1}}+CTR^{d-\frac{2sp}{p-1}}.\end{split}
\end{equation*}
Then, choosing $T = R^{2s}$, we have that
\begin{equation}\label{proof4}
\int_{|x|<R}u_0(x)\phi(x)\,dx\leq CR^{d-\frac{2s}{p-1}}.
\end{equation} 
If $p<p_{F},$ then $d-\frac{2s}{p-1}<0.$ Hence passing to the limit as $R\to\infty$ in \eqref{proof4} we deduce that $$\int_{\mathbb{R}^d}u_0(x)\,dx\leq 0,$$ which is a contradiction with $\int_{\mathbb{R}^d}u_0(x)\,dx>0.$

\subsubsection{The case $p<p_{F}$ and $\int_{\mathbb{R}^d}u_0(x)\,dx=0,\,u_0\not\equiv 0$} Let us rewrite inequality \eqref{proof2} for $T = R^{2s}$ and $f(x)\equiv 0:$
\begin{equation}\label{CC}\begin{split}
(1-2\varepsilon
)\int_{0}^\infty\int_{\mathbb{R}^d}|u|^p\varphi\,dx\,dt+\int_{|x|<R}u_0(x)\phi(x)\,dx&\leq CR^{d-\frac{2s}{p-1}}.\end{split}
\end{equation}
As $p<p_{F},$ then $d-\frac{2s}{p-1}<0.$ Hence, passing to the limit as $R\to\infty$ in the inequality above we have that \begin{equation*}\begin{split}
&(1-2\varepsilon
)\int_{0}^\infty\int_{\mathbb{R}^d}|u|^p\,dx\,dt+\int_{\mathbb{R}^d}u_0(x)\,dx\\&=(1-2\varepsilon
)\int_{0}^\infty\int_{\mathbb{R}^d}|u|^p\,dx\,dt\leq 0,\end{split}
\end{equation*} thanks to $$\int_{\mathbb{R}^d}u_0(x)\,dx=0.$$ Therefore, $u\equiv 0,\,\, a. e.\,\,(x,t)\in\mathbb{R}^d\times\mathbb{R}_+. $ This means that there is no global solution other than $u=0$.

\subsubsection{The case $p=p_{F}$} 
Let $p=p_F,$ then from \eqref{CC} follows that
\begin{align*}
\int_{0}^{T}\int_{|x|<R}|u|^p\varphi\,dx\,dt+\int_{|x|<R}u_0(x)\phi(x)\,dx&\leq C.
\end{align*} Hence, passing to the limit as $T\to\infty$ and $R\to\infty$ in the last inequality we have that
\begin{align*}
\int_{0}^{\infty}\int_{\mathbb{R}^d}|u|^p\,dx\,dt+\int_{\mathbb{R}^d}u_0(x)\,dx&\leq C.
\end{align*} 
As $\int_{\mathbb{R}^d}u_0(x)\,dx\geq 0,$ we conclude that $u\in L^p(\mathbb{R}_+\times\mathbb{R}^d).$

Using the H\"{o}lder's inequality to the right hand side of \eqref{proof1} we have
\begin{equation*}\begin{split}&
\int_{0}^{T/2}\int_{|x|<R/2}|u|^p\,dx\,dt+\int_{|x|<R}u_0(x)\phi(x)\,dx\\&\leq \mathcal{I}_1^{\frac{p-1}{p}}\left(\int_{T/2}^T\int_{|x|<R}|u|^p\,dx\,dt\right)^{\frac{1}{p}}+\mathcal{I}_2^{\frac{p-1}{p}}\left(\int_0^T\int_{|x|<R}|u|^p\,dx\,dt\right)^{\frac{1}{p}}.\end{split}
  \end{equation*}
Let $R=\Theta\mathcal{R},$ where $\Theta\in (1,\mathcal{R})$ and $\mathcal{R}$ is a large number. Since $p=p_F=1+\frac{2s}{d}$, setting estimates of the integrals $\mathcal{I}_1,$ $\mathcal{I}_2:$
\begin{equation*}\begin{split}
\mathcal{I}_1\leq C\Theta^d\mathcal{R}^dT^{1-\frac{p}{p-1}}\,\,\,\text{and}\,\,\,\mathcal{I}_2\leq CT\mathcal{R}^{d-\frac{2sp}{p-1}}\Theta^{d-\frac{2sp}{p-1}},\end{split}
\end{equation*}
from Subsubsection \ref{SSS1} we obtain
\begin{equation*}\begin{split}&
\int_{0}^{T/2}\int_{|x|<\Theta\mathcal{R}/2}|u|^p\,dx\,dt+\int_{|x|<\Theta\mathcal{R}}u_0(x)\phi(x)\,dx\\&\leq C(\Theta\mathcal{R})^{d\frac{p-1}{p}}T^{-\frac{1}{p}}\left(\int_{T/2}^T\int_{|x|<\Theta\mathcal{R}}|u|^p\,dx\,dt\right)^{\frac{1}{p}}\\&+CT^{\frac{p-1}{p}}(\Theta\mathcal{R})^{d\frac{p-1}{p}-2s}\left(\int_0^T\int_{|x|<\Theta\mathcal{R}}|u|^p\,dx\,dt\right)^{\frac{1}{p}}\\&\leq C(\Theta\mathcal{R})^{\frac{2ds}{d+2s}}T^{-\frac{d}{d+2s}}\left(\int_{T/2}^T\int_{|x|<\Theta\mathcal{R}}|u|^p\,dx\,dt\right)^{\frac{1}{p}}\\&+CT^{\frac{2s}{d+2s}}(\Theta\mathcal{R})^{\frac{2ds}{d+2s}-2s}\left(\int_0^T\int_{|x|<\Theta\mathcal{R}}|u|^p\,dx\,dt\right)^{\frac{1}{p}}.\end{split}
  \end{equation*}
Choosing $T=\mathcal{R}^{2s}$ we obtain
\begin{equation}\label{Aa}\begin{split}&
\int_{0}^{\mathcal{R}^s/2}\int_{|x|<\Theta\mathcal{R}/2}|u|^p\,dx\,dt +\int_{|x|<\Theta\mathcal{R}}u_0(x)\phi(x)\,dx\\&\leq C\Theta^{\frac{ds}{d+2s}}\left(\int_{{\mathcal{R}^s}/2}^{\mathcal{R}^s}\int_{|x|<\Theta\mathcal{R}}|u|^p\,dx\,dt\right)^{\frac{1}{p}}\\&+C\Theta^{-\frac{4s^2}{d+2s}}\left(\int_0^{\mathcal{R}^s}\int_{|x|<\Theta\mathcal{R}}|u|^p\,dx\,dt\right)^{\frac{1}{p}}.\end{split}
  \end{equation}
Since $u\in L_{loc}^p(\mathbb{R}_+\times \mathbb{R}^d)$, we obtain
\begin{equation*}\begin{split}&
\lim\limits_{\mathcal{R}\rightarrow \infty}\left(\int_{\mathcal{R}^{2s}/2}^{\mathcal{R}^{2s}}\int_{|x|<\Theta\mathcal{R}}|u|^p\,dx\,dt\right)^{\frac{1}{p}}\\&=\lim\limits_{\mathcal{R}\rightarrow \infty}\left(\int_{0}^{\mathcal{R}^{2s}}\int_{|x|<\Theta\mathcal{R}}|u|^p\,dx\,dt-\int_{0}^{\mathcal{R}^{2s}/2}\int_{|x|<\Theta\mathcal{R}}|u|^p\,dx\,dt\right)^{\frac{1}{p}}\\&=0.\end{split}
\end{equation*} 
As $$\int_{\mathbb{R}^d}u_0(x)\,dx> 0,$$ or $$\int_{\mathbb{R}^d}u_0(x)\,dx= 0,$$ letting $\mathcal{R}\rightarrow \infty$ and then $\Theta\rightarrow \infty$ we have that
$$\int_{0}^{\infty}\int_{\mathbb{R}^d}|u|^p\,dx\,dt\leq 0,$$ hence $u\equiv 0,$ which is a contradiction.
  \subsubsection{The case $p>p_{F}$} The proof of this part follows verbatim in the same manner as Part (ii) of Theorem \ref{thm:Main} with $f(x)\equiv 0.$
This complete the proof of this theorem. 

\subsection{Proof of Theorem \ref{thm:Main3}}
Assume that $p>1$ and let us rewritte \eqref{proof2} in the form
\begin{equation*}\begin{split}
R^{-\sigma}\int_{|x|<R}f(x)\phi(x)\,dx&\leq T^{-1}R^{-\sigma}\int_{|x|<R}u_0(x)\,dx\\&+ CR^{d-\sigma}T^{-\frac{p}{p-1}}+CR^{d-\sigma-\frac{2sp}{p-1}}.\end{split}
\end{equation*}
Then, passing to the limit as $R \rightarrow \infty$ in the last inequality, and taking into account assumption of Theorem \ref{thm:Main3}
$$\limsup_{R\rightarrow\infty}R^{-\sigma}\int_{|x|<R}f(x)\,dx>0,\,\sigma>d,$$
one deduces that
\begin{equation*}\begin{split}
\limsup_{R\rightarrow\infty}R^{-\sigma}\int_{|x|<R}f(x)\,dx&\leq 0,\end{split}
\end{equation*}
for any $T>0$, which is a contradiction. This means that there is no weak solution to problem \eqref{main} on $(0,T)\times\mathbb{R}^d.$ The proof is complete.

\section*{Declaration of competing interest}
	The Authors declare that there is no conflict of interest

\section*{Data Availability Statements} The manuscript has no associated data.

\end{document}